\documentclass[11pt]{article}
\usepackage[margin=1in]{geometry}
\usepackage{amsmath,amssymb,amsthm,mathtools,amsfonts}
\usepackage[colorlinks=true,linkcolor=blue,citecolor=blue,urlcolor=blue]{hyperref}
\usepackage{authblk}
\usepackage[sort]{cite}
\usepackage{tikz}
\usetikzlibrary{backgrounds,fit,positioning}
\usepackage{graphicx}

\usepackage{tikz}
\usepackage{pgfplots}
\pgfplotsset{compat=1.17}

\allowdisplaybreaks
\numberwithin{equation}{section}

\newtheorem{theorem}{Theorem}
\newtheorem{lemma}[theorem]{Lemma}
\newtheorem{corollary}[theorem]{Corollary}

\newtheorem{proposition}[theorem]{Proposition}
\theoremstyle{remark}

\newcommand \aln[2]
{
\begin{align}\label{#1}
#2
\end{align}
}

\newcommand\Abs[1]
{
\left | #1 \right |
}

\newcommand\Bra[1]
{
	\left ( #1 \right )
}

\newcommand \red[1] {{\color{red}{#1}}}

\newcommand{\s}{\mathcal A}
\newcommand{\N}{\mathbb N}
\newcommand{\Ccal}{\mathcal C}
\newcommand{\Fcal}{\mathcal F}
\newcommand{\z}{\boldsymbol z}
\newcommand{\w}{\boldsymbol w}
\newcommand{\st}{\operatorname{st}}
\newcommand{\ST}{\mathcal{ST}}
 \def \cala {{\cal A}}
 \def \mbfr {{\mathbb R}}
 
 \hypersetup{hidelinks}

\begin{document}
	\title{When chromatic polynomials coincide with list-color functions: a threshold linear in the maximum degree
	}
	
\author[1\thanks{Corresponding author. Email:
			meiqiaozhang95@163.com and meiqiaozhang@xmu.edu.cn.}]{Meiqiao Zhang}
		
\author[2\thanks{Email: fengming.dong@nie.edu.sg 
		and donggraph@163.com.}
	]{Fengming Dong}
	
	\affil[1]{\small School of Mathematical Sciences, Xiamen  University, China}

	\affil[2]{\small
		National Institute of Education,
		Nanyang Technological University, 
		Singapore}

	\date{}
	
	\maketitle

\begin{abstract}
Let $G$ be a simple graph with maximum degree $\Delta\ge 3$, and let $P(G,k)$ denote its chromatic polynomial. For each positive integer $k$, the list-color function $P_{\ell}(G,k)$ is the minimum number of $L$-colorings of $G$ over all $k$-assignments $L$. In this paper, we prove that $P_{\ell}(G,k)=P(G,k)$ for every integer $k\ge 23.41\Delta$. This gives a threshold for equality that is linear in the maximum degree and independent of the number of vertices or edges. It improves the known sufficient condition $k\ge |E(G)|-1$ for graphs with sufficiently many edges relative to their maximum degree.
\end{abstract}

\smallskip
\noindent \textbf{Keywords:} chromatic polynomial; list-color function; list-coloring

\smallskip
\noindent \textbf{Mathematics Subject Classification: 05C31, 05C30}


\section{Introduction}
In this article, we consider simple graphs only. 
For any graph $G$, let $V(G)$ and $E(G)$ be the vertex set and the edge set of $G$, respectively. For any non-empty subset $S$ of $V(G)$, let $G[S]$ denote the subgraph of $G$ induced by $S$. 
Denote by $\N$ the set of positive integers and by $\mbfr$ the set of real numbers. For any $k\in\N$, let $[k]=\{1,\dots,k\}$.
For any set $X$, we say $\z$ is a \textit{real-valued weight function} on $X$ if 
$\z$ is a 
mapping $X \rightarrow \mbfr$.

For any graph $G$, a 
{\it proper coloring} of $G$ is a mapping $\theta:V(G)\rightarrow \N$, such that $\theta(u)\neq \theta(v)$ for all $uv\in E(G)$.
For any $k\in \N$, a 
{\it proper $k$-coloring} of $G$ is a proper coloring $\theta$ such that $\theta(v)\in [k]$ for all $v\in V(G)$. 
Then the \textit{chromatic polynomial} $P(G, k)$ of $G$ is a polynomial that counts the number of proper $k$-colorings of $G$ for each $k \in \N$.
The chromatic polynomial was introduced by Birkhoff in~\cite{birk} as a tool for attacking the Four-Color Conjecture, but later became an important object of study in its own right due to its elegant properties.
More details on $P(G,k)$ can be found
in \cite{birk, birk2, dong1, dong0, Jackson2015, rea1, rea2}.

To generalize proper coloring, Vizing~\cite{viz}  and Erd\H{o}s, Rubin and Taylor~\cite{erdos} independently introduced the notion of list-coloring.
For any $k\in \N$,
a $k$-{\it assignment} of $G$ 
is a mapping $L:V(G)\rightarrow 2^{\N}$  such that
$|L(v)|=k$ for all $v\in V(G)$, and an $L$-\textit{coloring} of $G$ is a proper coloring $\theta$ with $\theta(v)\in L(v)$ for all $v\in V(G)$. 
Denote by $P(G,L)$ the number of $L$-colorings of $G$.
Then Kostochka and Sidorenko~\cite{kosto} introduced the \textit{list-color function} $P_\ell(G,k)$ of $G$ to
be the minimum value of $P(G, L)$  over all $k$-assignments $L$ of $G$ for each $k\in\N$. 
For more recent developments concerning $P_\ell(G,k)$, see \cite{Thom,Kaul22,Kaul23,Dong22,wangwei,Donner}.

Observe that there is a trivial $k$-assignment $L_0$ of $G$ with $L_0(v)=[k]$ for all $v\in V(G)$, which implies that $P(G,L_0)=P(G,k)$. Then by the definition of the list-color function, for each $k\in \N$,
\aln{listcolor}
{
P_\ell(G,k) \le P(G,k).
}
It is not difficult to verify that equality in (\ref{listcolor})  holds for 
any chordal graph $G$ and $k\in \N$~\cite{kosto}. Meanwhile, the inequality in (\ref{listcolor}) can be strict.
For example, 
$P(G,2)\ge 2$ holds for each bipartite graph $G$,
but $P_\ell(G,2)=0$ whenever $G$ contains $K_{2,4}$ as a subgraph.
Consequently, Kostochka and Sidorenko asked in~\cite{kosto} for which values of $k$ equality holds in (\ref{listcolor}). Donner first gave a surprising answer in~\cite{Donner}.

	\begin{theorem}[\cite{Donner}]\label{PLCG-Donner}
For any graph $G$, $P(G,k)=P_\ell(G,k)$ if $k$ is sufficiently large.
	\end{theorem}

Theorem~\ref{PLCG-Donner} reveals that the list-color function of every graph eventually coincides with its chromatic polynomial. 
Then the next question is naturally that for any graph $G$, what is the minimum integer $\tau(G)$ such that $P(G,k)=P_\ell(G,k)$ whenever $k\ge \tau(G)$.
 In 2009, Thomassen~\cite{Thom} established an upper bound for $\tau(G)$ in terms of the order of $G$.

	\begin{theorem}[\cite{Thom}]\label{PLCG-Thom}
For any graph $G$, $\tau(G)\le |V(G)|^{10}+1$.
	\end{theorem}

In 2017, Wang, Qian and Yan~\cite{wangwei} improved Theorem~\ref{PLCG-Thom} by introducing and applying expansion formulas of the chromatic polynomial and the list-color function in terms of edge subsets.

	\begin{theorem}[\cite{wangwei}]\label{PLCG-wangg}
For any graph $G$, $\tau(G)\le\frac{|E(G)|-1}{\log(1+\sqrt{2})}+1$.
	\end{theorem}
	
In 2023, Dong and Zhang~\cite{Dong22} obtained a further improvement by refining the techniques in~\cite{wangwei}.

\begin{theorem}[\cite{Dong22}]
For any graph $G$, $\tau(G)\le \max\{|E(G)|-1,1\}$.
\end{theorem}

At the end of \cite{Dong22}, Dong and Zhang conjectured that \(\tau(G)\) admits an upper bound linear in \(|V(G)|\), or even in the maximum degree \(\Delta(G)\). In this paper, we establish a bound linear in \(\Delta(G)\), thereby confirming the stronger form of this conjecture.

To establish the desired upper bound, it suffices to consider connected graphs. Indeed, if \(G\) has components \(G_1,\ldots,G_c\), then

$$
P(G,k)=\prod_{i=1}^{c}P(G_i,k)
\qquad\text{and}\qquad
P_{\ell}(G,k)
=\prod_{i=1}^{c}P_{\ell}(G_i,k).
$$
Consequently,
$$
\tau(G)=\max_{1\leq i\leq c}\tau(G_i).
$$

Moreover, every connected graph with maximum degree at most two is a path, including the single-vertex graph, or a cycle. For any graph $G$ which is either a path or a cycle, $P_{\ell}(G,k)=P(G,k)$ for every $k\in\N$
(see~\cite{Kirov2016,kosto}).
It therefore remains to consider connected graphs with maximum degree at least three, for which we prove the following quantitative result.

\begin{theorem}\label{main1}
Let $G$ be a connected graph with maximum degree $\Delta\ge 3$. For any $k$-assignment $L$ of $G$ with  $k\ge 23.41\Delta$,
\aln{main1-e1}
{
P(G,L) &\ge P(G,k)
 \exp\left(
\frac{ 10^{-6}}{k^2}
 \sum_{uv\in E(G)}|L(u)\setminus L(v)|\right).
}
Hence $P(G,k)=P_{\ell}(G,k)$ for every integer $k$ with $k\ge 23.41 \Delta$.
\end{theorem}

By Theorem~\ref{main1}, we improve the general upper bound on $\tau(G)$ from $ \max\{|E(G)|-1,1\}$ to $\max\{\lceil 23.41\Delta(G)\rceil,1\}$,
thereby replacing a worst-case $O(|V(G)|^2)$ sufficient bound by an $O(|V(G)|)$ bound.

To prove Theorem~\ref{main1}, we shall introduce expansion formulas in terms of vertex subsets for the chromatic polynomial and the list-color function in Section~\ref{sec2}. Then by further developing some results to compare the two formulas, we give the proof of Theorem~\ref{main1} in Section~\ref{sec3}. 

\section{Preliminary results
\label{sec2}}
In this section, we shall establish some preparatory results for proving Theorem~\ref{main1}.

Throughout the remainder of the paper, we assume that $k$ is a fixed positive integer,
$G=(V,E)$ is a connected graph with maximum degree $\Delta\ge 3$, and $L$ is a $k$-assignment of $G$. Let
$$
 \Ccal=\{S\subseteq V: |S|\ge2\text{ and }G[S]\text{ is connected}\}.
$$
For any $S\in\Ccal$, define
\aln{eq:c-beta}
 {
 c(S)=
 \sum_{\substack{A\subseteq E(G[S])\\(S,A)\text{ connected}}}
 (-1)^{|A|}.
 }
 For any $\cala\subseteq \Ccal$, 
 let $2^{\cala}_0$ be the family 
 of subsets $\Fcal$ of $\cala$ 
 with the property that 
 $S_1\cap S_2=\emptyset$
 for each pair of distinct members $S_1,S_2$ of $\Fcal$. 
Obviously, 
$\emptyset\in 2^{\cala}_0$,
and $\{S\}\in 2^{\cala}_0$
for every $S\in \cala$.

Moreover, let $\w_P$ and $\w_L$ be real-valued weight functions on $\Ccal$ such that for any $S\in\Ccal$,
$$
 \w_P(S)=\frac{c(S)}{k^{|S|-1}},
 \qquad
 \w_L(S)=\frac{c(S)}{k^{|S|}}
 \left|\bigcap_{v\in S}L(v)\right|.
$$
 Sokal~\cite{Sokal} established an expansion formula for the chromatic polynomial in terms of $\w_P$.
\begin{lemma}[Proposition 2.1 of~\cite{Sokal}]\label{lem1}
Let $G$ be a connected graph of order $n$. Then 
\aln{ep1-1}
{
P(G,k)= k^n
\sum_{\Fcal\in 2^{\Ccal}_0}
 \prod_{S\in\Fcal} \w_P(S),
}
where the product over an empty set is understood to be one.
\end{lemma}

For the sake of completeness, we give a proof of Lemma~\ref{lem1} below.

\begin{proof}
Let $G=(V,E)$ be a connected graph of order $n$. 
It is well known from \cite{Whi1932} that
\aln{PGk}
{P(G,k)=\sum_{A\subseteq  E}(-1)^{|A|}k^{q(A)},}
where $q(A)$ is the number of components of the graph $(V,A)$. 

Assume that $V$ is also a partition of $V$. Then we have
\aln{proof}{
P(G,k)=&\sum_{A\subseteq  E}(-1)^{|A|}k^{q(A)}\nonumber\\
=&\sum_{s\ge 1}\sum_{\substack{A\subseteq  E\\q(A)=s}}(-1)^{|A|}k^{s}\nonumber\\
=&\sum_{s\ge 1}\sum_{\substack{A\subseteq  E\\q(A)=s}}\prod_{\substack{H~\text{is a component}\\ \text{of}~(V,A)}}(-1)^{|E(H)|}k\nonumber\\
=&\sum_{s\ge 1}\sum_{\substack{ V_1,V_2,\dots,V_s ~\text{partition}~V\\G[V_1], G[V_2],\dots,G[V_s]~\text{connected}}}\prod_{i=1}^s \sum_{\substack{E_i\subseteq E(G[V_i])
\\(V_i,E_i)~\text{connected}}} \left((-1)^{|E_i|}k\right)\nonumber\\
=&k^n\sum_{s\ge 1}\sum_{\substack{ V_1,V_2,\dots,V_s ~\text{partition}~V\\G[V_1], G[V_2],\dots,G[V_s]~\text{connected}}}\prod_{i=1}^s
\sum_{\substack{E_i\subseteq E(G[V_i])
		\\(V_i,E_i)~\text{connected}}}
 \left((-1)^{|E_i|}k^{1-|V_i|}\right)\nonumber\\
=&k^n\sum_{s\ge 1}\sum_{\substack{ V_1,V_2,\dots,V_s ~\text{partition}~V\\G[V_1], G[V_2],\dots,G[V_s]~\text{connected}}}\prod_{1\le i\le s\atop |V_i|\ge 2}k^{1-|V_i|}
\sum_{\substack{E_i\subseteq E(G[V_i])
\\(V_i,E_i)~\text{connected}}}
 \left((-1)^{|E_i|}\right)\nonumber\\
=&k^n\sum_{s\ge 1}\sum_{\substack{ V_1,V_2,\dots,V_s ~\text{partition}~V\\G[V_1], G[V_2],\dots,G[V_s]~\text{connected}}}\prod_{1\le i\le s\atop |V_i|\ge 2} \w_P(V_i)\nonumber\\
=&k^n\sum_{\Fcal \in 2^{\Ccal}_0}
 \prod_{S\in\Fcal} \w_P(S).
}
\end{proof}

Further, Wang, Qian and Yan obtained in~\cite{wangwei} the following extension of (\ref{PGk}) for list-colorings that
\aln{PGL}{P(G,L)=\sum_{A\subseteq  E} (-1)^{|A|}\prod_{\substack{H~\text{is a component}\\ \text{of}~(V,A)}}\left(
	\Abs{\bigcap_{v\in V(H)} L(v)}
	\right).}
Then by applying the same idea as in (\ref{proof}) to (\ref{PGL}), we can derive an analogous expansion formula for counting list-colorings in terms of $\w_L$.

\begin{lemma}\label{lem2}
Let $G$ be a connected graph of order $n$. If $L$ is a $k$-assignment of $G$, then
\aln{ep1-2}
{
P(G,L)=k^n
\sum_{\Fcal \in 2^{\Ccal}_0}
 \prod_{S\in\Fcal} \w_L(S).
}
\end{lemma}

Therefore, a natural strategy for proving Theorem~\ref{main1} is to compare $P(G,k)$ and $P(G,L)$ by (\ref{ep1-1}) and (\ref{ep1-2}). Following this approach, we further introduce several helpful lemmas in the following.

For any connected graph $H$, let $\ST(H)$ be the set of spanning trees of $H$ and let $\st(H)=|\ST(H)|$. 
By Whitney's broken cycle theorem~\cite{Whi1932} and (\ref{eq:c-beta}), for any $S\in\Ccal$, 
\aln{eq:c-tree}
{
 |c(S)|= \left| \sum_{\substack{A\subseteq E(G[S])\\(S,A)\text{ connected}}} (-1)^{|A|}\right|\le  \sum_{\substack{A\subseteq E(G[S])\\(S,A)\in \ST(G[S])}} 1=  \st(G[S]).
}
A more general version of (\ref{eq:c-tree}) was also given in~\cite{penrose,Sokal}.

Further, an upper bound for the number of subtrees with a fixed order and containing a particular vertex was determined in~\cite{Sokal}.
 \begin{lemma}[\cite{Sokal}]\label{lem:vertextrees}
For any $v\in V$ and $s\in \N$, let $N_v(s)$ be the number of trees $T$ in $G$ such that $|V(T)|=s$ and $v\in V(T)$. Then $N_v(1)=1$ and for $s\ge 2$,
\aln{eq2-6}{
N_v(s)=\sum_{
	S\in\Ccal: v\in S\atop |S|=s
}
\st(G[S]).
}
Moreover,
\aln{eq:vertex-tree}
{
N_v(s)
 \le\frac{(s\Delta)^{s-1}}{s!}
 \le\Delta^{s-1}\exp(s-1).
}
\end{lemma}

\begin{proof}
It is clear that (\ref{eq2-6}) follows from the definition.
For (\ref{eq:vertex-tree}),
the first inequality was
 established in Proposition 4.3(f) of
\cite{Sokal}.
Thus it suffices to prove the second inequality of (\ref{eq:vertex-tree}). When $s=1$, the result is trivial. Assume that $s\ge 2$ in the following. 

Note that for all $n\in\N$, $\exp(n)=\sum_{i=0}^\infty n^i/i!\ge n^n/n!$, and
$$n^{n}=\left(\frac{(n+1)(n-1)+1}{n}\right)^n\ge (n+1)^{n-1},$$
where the last inequality follows from the AM-GM inequality.
 Thus
\aln{le8-e3}
{
 \frac{s^{s-1}}{s!}
 =
 \frac{s^{s-2}}{(s-1)!}
 \le
 \frac{(s-1)^{s-1}}{(s-1)!}
 \le \exp(s-1),
}
which completes the proof.
\end{proof}

Based on Lemma~\ref{lem:vertextrees}, we also provide an upper bound for the number of subtrees with a fixed order and containing a particular edge.
For any $S\in \Ccal$ 
and any edge $e$ in $G$, 
let $\ST_e(G[S])$
denote the set of spanning trees $T$ of $G[S]$ with $e\in E(T)$.

 \begin{lemma}\label{lem:edge-trees}
For any $e\in E$ and $s\in\N$ with $s\ge 2$, let $N_e(s)$ be the number of trees $T$ 
in $G$
such that $|V(T)|=s$ and $e\in E(T)$. Then
\aln{eq2-8}
{
N_e(s)=\sum_{\substack{S\in\Ccal\\|S|=s}}
\sum_{T\in\ST_e(G[S])} 1.
}
Moreover,
\aln{eq:edge-tree}
{
 N_e(s)\le\frac{2s^{s-3}}{(s-2)!}\Delta^{s-2}
 \le2\Delta^{s-2}\exp(s-2).
}
\end{lemma}

\begin{proof}
It is clear that (\ref{eq2-8}) follows from the definition.
Then we shall prove (\ref{eq:edge-tree}) by applying Lemma~\ref{lem:vertextrees}. 

For all $i\in\N$, 
let
$a_i=\frac{i^{i-1}}{i!}
=\frac{i^{i-2}}{(i-1)!}.
$
Then by Lemma~\ref{lem:vertextrees}, 
for any $v\in V$,
\aln{ine2-19}
{ 
N_v(s)\le a_{s}\Delta^{s-1}.
} 
Also, it was proved in Section 6 of~\cite{Stone1985} that for all $s\ge 2$, 
$$
\frac{1}{2}\sum_{i=1}^{s-1}\binom{s-2}{i-1}i^{i-2}(s-i)^{s-i-2}=s^{s-3},
$$
which implies that
\aln{ine2-18}
{ 
\sum_{i=1}^{s-1}a_i a_{s-i}
 =
 \frac{2s^{s-3}}{(s-2)!}.
}

Let $e=uv\in E$.
For any tree $T$ with $|V(T)|=s$ and $uv\in E(T)$, observe that deleting $uv$ from $T$ yields two trees, one containing $u$ and the other one containing $v$, where the sum of the orders of the two trees is $|V(T)|$. Thus
$$
N_e(s)
 \le
 \sum_{i=1}^{s-1}N_u(i)N_v(s-i)
\le 
  \Delta^{s-2}
  \sum_{i=1}^{s-1}a_i a_{s-i}= \frac{2s^{s-3}}{(s-2)!}\Delta^{s-2},
$$
 where the last inequality and equality follow from (\ref{ine2-19}) and (\ref{ine2-18}), respectively.
 
Now it remains to show the last inequality in (\ref{eq:edge-tree}). 
It can be easily verified when $s=2,3$, while for $s\ge 4$,
\aln{}
{ 
 \frac{s^{s-3}}{(s-2)!}=&  \frac{(2+s-2)^{s-3}}{(s-2)!}\nonumber\\
 =& \frac{1}{(s-2)!}\sum_{t=0}^{s-3}\frac{(s-3)!}{t!(s-3-t)!}2^t(s-2)^{s-3-t}\nonumber\\
 =& \sum_{t=0}^{s-3}\frac{2^t}{(s-2)\times t!}\frac{(s-2)^{s-3-t}}{(s-3-t)!}\nonumber\\
\le & \sum_{t=0}^{s-3}\frac{2^t}{2\times t!}\frac{(s-2)^{s-3-t}}{(s-3-t)!}\nonumber\\
\le & \sum_{t=0}^{s-3}\frac{(s-2)^{s-3-t}}{(s-3-t)!}\nonumber\\
=&
\sum_{j=0}^{s-3}\frac{(s-2)^j}{j!}\nonumber\\
< &
\exp(s-2).
} 
The proof is complete.
\end{proof}

To conclude this section, we shall derive some numerical results based on Lemma~\ref{lem:edge-trees}.
For any $uv\in E$, it is clear that $|L(u)\setminus L(v)|=|L(v)\setminus L(u)|$. Then let 
\aln{alpha}
{\alpha_L(uv)=|L(u)\setminus L(v)|~~~\text{ and }~~~
D=\sum_{e\in E}\alpha_L(e).
}
We also need the following result established in~\cite{wangwei}.

\begin{lemma}[\cite{wangwei}]\label{wang}
For any $S\in\Ccal$ and  $T\in\ST(G[S])$,
\aln{eq:defect-tree}
{
 k-\left|\bigcap_{v\in S}L(v)\right|\le\sum_{e\in E(T)}\alpha_L(e).
}
\end{lemma}

\begin{proposition}\label{prop1}
For $2\le s\le |V|$,
\aln{eq:defect-large}
{
 \sum_{\substack{S\in\Ccal\\|S|
 		=s}}
\Bra{
	1-\frac 1k 
\Abs{\bigcap_{v\in S}L(v) }
}
|\w_P(S)|
 \le\frac{2D \Delta^{s-2}}{k^s}\exp(s-2).
}
In particular,
\aln{eq:defect-two}
{
 \sum_{\substack{S\in\Ccal\\|S|
 		=2}}
 \Bra{
 	1-\frac 1k 
 	\Abs{\bigcap_{v\in S}L(v) }
 }	
\w_P(S)
=-\frac{D}{k^2}.
}
\end{proposition}

\begin{proof}
We shall prove \eqref{eq:defect-two} first. It is clear that any $S\in\Ccal$ with $|S|=2$ is in fact $\{u,v\}$ for some $uv\in E$. Then obviously,
$c(S)=-1$ and $\w_P(S)=-1/k$. Thus
\aln{prop1-e3}
{
 \sum_{\substack{S\in\Ccal\\|S|=2}}
 \Bra{
 	1-\frac 1k 
 	\Abs{\bigcap_{v\in S}L(v) }
 }
\w_P(S)
=&-\frac{1}{k}\sum_{uv\in E}\left(\frac{k-|L(u)\cap L(v)|}{k}\right)\nonumber\\
=&-\frac{1}{k}\sum_{uv\in E}\frac{|L(u)\setminus L(v)|}{k}\nonumber\\
=&-\frac{1}{k^2}\sum_{uv\in E}\alpha_L(uv)\nonumber\\
=&-\frac{D}{k^2}.
}

For (\ref{eq:defect-large}), now assume $|S|=s\ge 2$.
By (\ref{eq:c-tree}),
$$
|\w_P(S)|=\frac{|c(S)|}{k^{|S|-1}}\le \frac{\st(G[S])}{k^{s-1}}.
$$
Then
\aln{}
{
 k^s \sum_{\substack{S\in\Ccal\\|S|=s}}
\left(
1-\frac 1k 
\left|\bigcap_{v\in S}L(v)\right|
\right)
|\w_P(S)|
\le& 
 \sum_{\substack{S\in\Ccal\\|S|=s}}
\left(k-\left|\bigcap_{v\in S}L(v)\right|\right)\st(G[S])\nonumber\\
=& 
 \sum_{\substack{S\in\Ccal\\|S|=s}}\sum_{T\in\ST(G[S])}
\left(k-\left|\bigcap_{v\in S}L(v)\right|\right)\nonumber\\
 \le&
\sum_{\substack{S\in\Ccal\\|S|=s}} \sum_{T\in\ST(G[S])}
 \sum_{e\in E(T)}\alpha_L(e)\nonumber\\
 =&
 \sum_{e\in E}\alpha_L(e)
 \sum_{\substack{S\in\Ccal\\|S|=s}}\sum_{T\in\ST_e(G[S])}1\nonumber\\
 = & 
 \sum_{e\in E}\alpha_L(e)N_e(s)\nonumber\\
 \le & 2\Delta^{s-2}\exp(s-2) \, D,
}  
where the second and the last inequalities
follow
from \eqref{eq:defect-tree} and Lemma~\ref{lem:edge-trees},
respectively.
\end{proof}

\section{Proof of Theorem~\ref{main1}
\label{sec3}}
In this section, we prove Theorem~\ref{main1}. Our approach is to compare (\ref{ep1-1}) and (\ref{ep1-2}) by constructing an auxiliary polynomial that includes both expressions as special cases and then establishing the desired conclusion through  monotonicity arguments.

Let 
$\z$ be a real-valued weight function 
on $\Ccal$ and $\s\subseteq \Ccal$.
Define
\aln{eq:Z}
{
\phi_\s(\z)=
 \sum_{\Fcal \in 2^{\s}_0
              }
 \prod_{S\in\Fcal}\z(S).
}

\subsection{A result
associated with $\phi_\s(\z)$
}

By Lemmas~\ref{lem1} and~\ref{lem2}, it is clear that $P(G,k)=k^n\phi_\Ccal(\w_P)$ and $P(G,L)=k^n\phi_\Ccal(\w_L)$. Moreover,
 for any $S\in\Ccal$, let
$$
N_{\s}[S]=\{S'\in\s:S'\cap S\ne\varnothing\}.
$$
Then for each $S\in\s$, $S\in N_{\s}[S]$, and 
\aln{eq:deletion}
 {
\phi_\s(\z)=\phi_{\s\setminus\{S\}}(\z)
 +\z(S)\phi_{\s\setminus N_\s[S]}(\z),
}
which further gives
\aln{eq:derivative}
{
\frac{\partial \phi_\s(\z)}
{\partial \z(S)}=\phi_{\s\setminus N_\s[S]}(\z),
}
where $\partial \phi_\s(\z)
=\phi_\s(\z)-\phi_{\s\setminus\{S\}}(\z)$ and $\partial \z(S)=\z(S)$.

The first step towards proving Theorem~\ref{main1} is to provide a sufficient condition on $\z$ for $\phi_\s(\z)\neq0$, as follows. 

\begin{lemma}\label{lem:ratio}
Suppose that $\z$ is a real-valued weight function on $\Ccal$, and for each $S\in \Ccal$,  $\rho_S, a_S$ are non-negative real numbers such that $|\z(S)|\le\rho_S$ and
\aln{eq:local-condition}
{
 \sum_{S'\in N_{\Ccal}[S]}
 \rho_{S'}\exp(a_{S'})\le a_S.
}
Then for any $\s\subseteq \Ccal$,
\begin{enumerate}
\item $\phi_{\s}(\z)\neq 0$,
\item for any $S\in\s$, 
$
\left|
\frac{\phi_{{\s}\setminus \{S\}}(\z)}{\phi_{\s}(\z)}-1
 \right|
 \le \rho_S\exp(a_S)
 $ and 
 thus 
 $\left|
 \frac{\phi_{{\s}\setminus \{S\}}(\z)}{\phi_{\s}(\z)}
 \right|
 \le 1+\rho_S\exp(a_S),
 $
\item for any $S\in\s$, 
$ \left|
\frac{\phi_{\s\setminus N_{\s}[S]}(\z)}
{\phi_{\s}(\z)}-1
\right|
\le\exp(a_S)-1$,
and thus
$
 \left|
\frac{\phi_{\s\setminus N_{\s}[S]}(\z)}
{\phi_{\s}(\z)}
\right|
\le\exp(a_S).
$
\end{enumerate}
\end{lemma}

\begin{proof}
In this proof, we omit the notation $\z$ when no confusion can arise.
For each $S\in\Ccal$, let $\mu_S=\rho_S\exp(a_S)$. Then $\mu_S\ge 0$. 

We shall first prove (i) and (ii) by induction on $|\s|$. 
It is trivial when $|\s|=0$. 

Now assume that $|\s|\ge 1$,
and both (i) and (ii) hold
for any $\s'\subseteq \Ccal$ with $|\s'|<|\s|$.

Let $S\in \s$. Then $S\in  N_\s[S]$.
Since both $|\s\setminus\{S\}|$ and $|\s\setminus N_\s[S]|$ are smaller than $|\s|$, we have $\phi_{\s\setminus\{S\}}\ne 0$
and $\phi_{\s\setminus N_\s[S]}\neq 0$ 
by induction. Let 
$$
 R=\z(S)
 \frac{\phi_{\s\setminus N_\s[S]}}
      {\phi_{\s\setminus\{S\}}}.
$$
Then by \eqref{eq:deletion},  
\aln{eq1-10}
{
\phi_\s=\phi_{\s\setminus\{S\}}(1+R).
}
Assume that $N_{\s}[S]=
\{S_0, S_1, \ldots, S_m\}$,
where $S_0=S$.
By repeatedly deleting elements in $N_\s[S]$ from $\s\setminus\{S\}$ and applying the induction hypothesis of (ii),
\aln{eq1-2}
{
 |R|=|\z(S)|
  \left|\frac{\phi_{\s\setminus N_\s[S]}}
       {\phi_{\s\setminus\{S\}}}\right|
  \le   \rho_S  
  \prod_{i=0}^{m-1}
  \Abs{
  \frac{\phi_{\s\setminus 
\{S_0,\dots,S_{i+1}\}  
}}
  {\phi_{\s\setminus\{S_0,\dots,S_i\}}}
}
       \le \rho_S     
\prod_{S'\in N_{\s}[S]
	\atop S'\ne S
}
(1+\mu_{S'}).
}
Note that
\eqref{eq:local-condition} implies that
\aln{eq:product-condition}{
\prod_{S'\in N_\s[S]}(1+\mu_{S'})
\le  \prod_{S'\in N_\s[S]}\exp(\mu_{S'})
\le \exp\left(\sum_{S'\in N_\Ccal[S]}\mu_{S'}\right)
 \le\exp(a_S),
}
where the second inequality holds as $N_\s[S]\subseteq N_\Ccal[S]$ and $\mu_{S'}\ge0$ for any $S'\in\Ccal$.

Since $\mu_S\ge 0$, (\ref{eq1-2}) and (\ref{eq:product-condition}) together give
\aln{ieq1-1}
{
 |R|\le \rho_S
\frac{  \prod_{S'\in N_\s[S]}(1+\mu_S')}{1+\mu_S}
 \le \frac{\rho_S\exp(a_S)}{1+\mu_S}
=\frac{\mu_S}{1+\mu_S}<1.
}

Clearly, (\ref{ieq1-1}) indicates $|1+R|\ge 1-|R|>0$. 
Since $|\phi_{\s\setminus\{S\}}|>0$, 
 by (\ref{eq1-10}), 
$$|\phi_\s|
=|\phi_{\s\setminus\{S\}}||1+R|> 0,$$ 
implying that $\phi_\s\neq 0$.
Moreover, for any $S\in\s$, 
by (\ref{eq1-10}),
\aln{ieq1-2}
{
 \left|\frac{\phi_{\s\setminus\{S\}}}{\phi_\s}-1\right|
=\Abs{\frac 1{R+1}-1}
 =\frac{|R|}{|1+R|}\le\frac{|R|}{1-|R|}\le \frac{\frac{\mu_S}{1+\mu_S}}{1-\frac{\mu_S}{1+\mu_S}}=\mu_S.
}
Hence by induction, (i) and (ii) hold.

Now it remains to prove (iii). 
For any $S\in\s$, assume that $
N_{\s}[S]=\{S_1,\ldots,S_m\}$. Let $\s_0=\s$, and $\s_i=\s_{i-1}\setminus\{S_i\}$ for each $i\in[m]$. Then $\s_m=\s\setminus N_{\s}[S]$. 

Further, let $\varepsilon_i
=\frac{\phi_{\s_i}}{\phi_{\s_{i-1}}}-1$ for each $i\in[m]$. Then $
\frac{\phi_{\s\setminus N_{\s}[S]}}{\phi_{\s}}
=\prod\limits_{i=1}^m(1+\varepsilon_i)$, and $|\varepsilon_i|\le\mu_{S_i}$ follows from (ii).
     Thus
\aln{}{
\left|
\frac{\phi_{\s\setminus N_{\s}[S]}}
     {\phi_{\s}}
-1
\right|
&=
\left|
\prod_{i=1}^m(1+\varepsilon_i)-1
\right|\nonumber\\
&=
\left|
\sum_{\varnothing\neq I\subseteq[m]}
\prod_{i\in I}\varepsilon_i
\right|\nonumber\\
&\le
\sum_{\varnothing\neq I\subseteq[m]}
\prod_{i\in I}|\varepsilon_i|
\nonumber\\
&=
\prod_{i=1}^m(1+|\varepsilon_i|)-1\nonumber\\
&\le
\prod_{S'\in N_{\s}[S]}(1+\mu_{S'})-1\nonumber\\
&\le \exp(a_S)-1,
}
where the last inequality holds due to (\ref{eq:product-condition}).
The proof is complete.
\end{proof}

Next, we specify a real-valued weight function on $\Ccal$ which satisfies the requirements of Lemma~\ref{lem:ratio}.
 For $0\le t\le1$, let $\w_t$ be the real-valued weight function on $\Ccal$ that for each $S\in \Ccal$,
 \aln{eq:interpolation}
 {
\w_t(S)=\w_P(S)\left(1-t\left(1-\frac{1}{k}\left|\bigcap_{v\in S}L(v)\right|\right)\right).
}
 Then $\w_0(S)=\w_P(S)$, $\w_1(S)=\w_L(S)$, and
 $|\w_t(S)|\le|\w_P(S)|$ for $0\le t\le1$.
We shall show in the next subsection that for $0\le t\le 1$, $\phi_\Ccal(\w_t)\neq 0$ under some given conditions.

\subsection{Proof of Theorem~\ref{main1}}

In the following, we first
establish two results
(i.e., Lemma~\ref{lem:uniform} and Proposition~\ref{final})
with assumption that 
 $r,\lambda,x,\epsilon$ are positive real numbers satisfying the following conditions:
\begin{equation}
	\label{conditions}
\left\{
\begin{aligned}
&0<x<1,
\\
&\exp(\lambda)
\cdot \frac{x}{1-x}\leq \lambda,\\
&\exp(2\lambda)
\cdot \frac{1+x}{1-x}
\le 2-\epsilon,\\
&
r=\frac{\exp(1+\lambda)}{x}.
\end{aligned}
\right.
\end{equation}
Then in the proof of Theorem~\ref{main1} on page~\pageref{pro-main1}, 
suitable values of $r$, $\lambda$, $x$, and $\epsilon$ satisfying all conditions in~(\ref{conditions})
will be chosen.


\begin{lemma}
	\label{lem:uniform}
Let	$r,\lambda,x$ 
and $\epsilon$ be any positive real numbers satisfying
	all conditions in 
	(\ref{conditions}).
If $k\ge r\Delta$ and $\Delta\ge 3$, then for any $0\le t\le 1$, we have $\phi_\Ccal(\w_t)\ne 0$,
and for any
$S\in\Ccal$,
\aln{uniform-e1}
{ \left|\frac{\phi_{\Ccal\setminus N_\Ccal[S]}( \w_t)}
      {\phi_\Ccal( \w_t)}-1\right|\le
      \exp\left(\lambda|S|\right)-1,
}
implying that 
 $
      \left|\frac{\phi_{\Ccal\setminus N_\Ccal[S]}( \w_t)}
            {\phi_\Ccal( \w_t)}\right|\le\exp
            \left(\lambda|S|\right).
$
\end{lemma}

\begin{proof}

To prove this lemma, it suffices to apply Lemma~\ref{lem:ratio} by taking $\rho_S=|\w_P(S)|$ and  $a_S=\lambda|S|$ for each $S\in\Ccal$. Since $|\w_t(S)|\le |\w_P(S)|=\rho_S$, it remains to verify (\ref{eq:local-condition}). 

For any $S\in \Ccal$,
\aln{ine3-15}{
 \sum_{Q\in N_{\Ccal}[S]}
 \rho_Q\exp(a_Q)
 &= \sum_{\substack{Q\in\Ccal\\Q\cap S\ne\varnothing}}
 |\w_P(Q)|\exp(\lambda|Q|) 
 \nonumber \\
 &\le
   \sum_{v\in S}
   \sum_{\substack{Q\in\Ccal:v\in Q}}
 |\w_P(Q)|\exp(\lambda|Q|).
}

For any fixed $v\in V$, by (\ref{eq:c-tree}), (\ref{conditions}), and Lemma~\ref{lem:vertextrees},
\aln{eq:vertex-sum}
{
 \sum_{\substack{Q\in\Ccal:v\in Q}}
 |\w_P(Q)|\exp(\lambda|Q|)
&=\sum_{\substack{Q\in\Ccal:v\in Q}}
  \frac{|c(Q)|}{k^{|Q|-1}}\exp(\lambda|Q|)\nonumber\\
  &\le \sum_{\substack{Q\in\Ccal:v\in Q}}
    \frac{\st(G[Q])}{k^{|Q|-1}}\exp(\lambda|Q|)\nonumber\\
  &=\sum_{s\ge 2} \frac{\exp(\lambda s)}{k^{s-1}}  \sum_{\substack{Q\in\Ccal:v\in Q\\ |Q|=s}}
\st(G[Q])\nonumber\\
  &\le \sum_{s\ge 2} \frac{\exp(\lambda s)}{k^{s-1}}  \Delta^{s-1}\exp(s-1)\nonumber\\
 & \le \sum_{s\ge 2} \left(\frac{1}{r}\right)^{s-1} \exp(\lambda )\left(\exp(\lambda+1)\right)^{s-1}\nonumber\\
 &=\exp(\lambda)\sum_{i\ge 1}x^i
 \nonumber \\
 &=\exp(\lambda) \cdot \frac{x}{1-x}
 \nonumber \\ 
 &\le \lambda, 
  }
where the last four steps
follow from $k\ge r\Delta$  and the conditions
in (\ref{conditions}).

Then by (\ref{conditions}), (\ref{ine3-15}) and  
(\ref{eq:vertex-sum}), 
$$
 \sum_{Q\in N_{\Ccal}[S]}
 \rho_Q\exp(a_Q)
 \le 
  \sum_{v\in S}\sum_{Q\in \Ccal:v\in Q} 
  |\w_P(Q)|\exp(\lambda|Q|)
 \le
 \sum_{v\in S} \lambda
 =a_S.
$$
Hence the result follows from Lemma~\ref{lem:ratio} (i) and (iii).
\end{proof}

\begin{proposition}\label{final}
Let	$r,\lambda,x$ 
	and $\epsilon$ be any positive real numbers satisfying
	all conditions in 
	(\ref{conditions}).
If $k\ge r\Delta$ and $\Delta\ge 3$, then 
\aln{finalieq}
{P(G,L)\ge 
	P(G,k) \exp\left(\frac{\epsilon}{k^2}\sum_{uv\in E(G)}|L(u)\setminus L(v)|\right).}
\end{proposition}
\begin{proof}
By the conditions in (\ref{conditions}) and $\lambda>0$, we have $r=\exp(1+\lambda)/x>1$. Then $k\ge r\Delta$ implies that $k\ge \Delta+1$ as $k$ is an integer. Hence by Brooks' Theorem, $P(G,k)>0$.

Let $\varphi(t)$ be the polynomial for $0\le t\le 1$ such that 
$\varphi(t)=\phi_\Ccal(\w_t),$
where $\phi_\Ccal(\w_t)$ is defined in (\ref{eq:Z}) and (\ref{eq:interpolation}). Then $\varphi(0)=P(G,k)/k^n>0$ and $\varphi(1)=P(G,L)/k^n$. By (\ref{alpha}), it suffices to prove that $\varphi(1)\ge\varphi(0) \exp(D \epsilon/k^2)$.

It is clear that $\varphi(t)$ is real-valued and continuous, and
Lemma~\ref{lem:uniform} shows that $\varphi(t)=\phi_\Ccal(\w_t)\neq 0$ for $0\le t\le1$. 
Then $\varphi(0)>0$ implies that $\varphi(t)>0$ for $0\le t\le 1$.

By (\ref{eq:derivative}) and (\ref{eq:interpolation}), for each $S\in\Ccal$,
$$
\left\{
\begin{aligned}
	\frac{\partial \phi_\Ccal(\w_t)}{\partial \w_t(S)}&=\phi_{\Ccal\setminus N_\Ccal[S]}(\w_t),\\
	\frac{d\w_t(S)}{dt}&=-\w_P(S)\left(1-\frac{1}{k}\left|\bigcap_{v\in S}L(v)\right|\right).\\
\end{aligned}
\right.
$$
Then
\aln{}{
	\varphi'(t)
	&=
	\sum_{S\in\mathcal C}
	\frac{\partial \phi_\Ccal(\w_t)}{\partial \w_t(S)}
	\frac{d\w_t(S)}{dt}
	\nonumber\\
	&=
	-\sum_{S\in\mathcal C}\phi_{\mathcal C\setminus N_{\mathcal C}[S]}(\w_t)\w_P(S)\left(1-\frac{1}{k}\left|\bigcap_{v\in S}L(v)\right|\right).
}

Since $\varphi(t)>0$ for $0\le t\le 1$, we further have
$$
\frac{\varphi'(t)}{\varphi(t)}=
-\sum_{S\in\mathcal C}\frac{
	\phi_{\mathcal C\setminus N_{\mathcal C}[S]}(\w_t)
}{
	\phi_{\mathcal C}(\w_t)
}\w_P(S)\left(1-\frac{1}{k}\left|\bigcap_{v\in S}L(v)\right|\right).
$$
Let 
$$Q_1=\left|-\frac{D}{k^2}-\sum_{\substack{S\in\mathcal C\\ |S|=2}}\frac{
	\phi_{\mathcal C\setminus N_{\mathcal C}[S]}(\w_t)
}{
	\phi_{\mathcal C}(\w_t)
}\w_P(S)\left(1-\frac{1}{k}\left|\bigcap_{v\in S}L(v)\right|\right)\right|,$$
and 
$$Q_2=\left|-\sum_{\substack{S\in\Ccal\\ |S|\ge3}}
\frac{
	\phi_{\mathcal C\setminus N_{\mathcal C}[S]}(\w_t)
}{
	\phi_{\mathcal C}(\w_t)
}\w_P(S)\left(1-\frac{1}{k}\left|\bigcap_{v\in S}L(v)\right|\right)\right|.$$
Then 
\aln{ine3-21}
{\left|\frac{\varphi'(t)}{\varphi(t)}-\frac{D}{k^2}
	\right|\le Q_1+Q_2.}
We next estimate $Q_1$ and $Q_2$.

By Proposition~\ref{prop1} and Lemma~\ref{lem:uniform},
\aln{boundq1}{
	Q_1
	&=
	\left|
	-\frac{D}{k^2}
	-
	\sum_{uv\in E}
	\frac{
		\phi_{\mathcal C\setminus N_{\mathcal C}[\{u,v\}]}(\boldsymbol w_t)
	}{
		\phi_{\mathcal C}(\boldsymbol w_t)
	}
	\boldsymbol w_P(\{u,v\})\left(1-\frac{1}{k}\left|L(u)\cap L(v)\right|\right)
	\right|
	\notag\\
	&=
	\left|
	-\frac{1}{k^2}\sum_{uv\in E}\alpha_L(uv)
	+\frac{1}{k^2}
	\sum_{uv\in E}
	\frac{
		\phi_{\mathcal C\setminus N_{\mathcal C}[\{u,v\}]}(\boldsymbol w_t)
	}{
		\phi_{\mathcal C}(\boldsymbol w_t)
	}
	\alpha_L(uv)\right|
	\notag\\
	&=
	\frac{1}{k^2}
	\left|
	\sum_{uv\in E}
	\alpha_L(uv)
	\left(
	\frac{
		\phi_{\mathcal C\setminus
			N_{\mathcal C}[\{u,v\}]}(\boldsymbol w_t)
	}{
		\phi_{\mathcal C}(\boldsymbol w_t)
	}-1
	\right)
	\right|
	\notag\\
	&\le
	\frac{1}{k^2}
	\sum_{uv\in E}
	\alpha_L(uv)
	\left|
	\frac{
		\phi_{\mathcal C\setminus
			N_{\mathcal C}[\{u,v\}]}(\boldsymbol w_t)
	}{
		\phi_{\mathcal C}(\boldsymbol w_t)
	}
	-1
	\right|
	\notag\\
	&\le
	\frac{D}{k^2}\bigl(\exp(2\lambda)-1\bigr),
}
and
\aln{eq:error-three}{
	Q_2
	&\le
	\sum_{s\ge3}
	\sum_{\substack{S\in\mathcal C\\ |S|=s}}
	\left(1-\frac{1}{k}\left|\bigcap_{v\in S}L(v)\right|\right)|w_P(S)|
	\left|
	\frac{
		\phi_{\mathcal C\setminus N_{\mathcal C}[S]}(\boldsymbol w_t)
	}{
		\phi_{\mathcal C}(\boldsymbol w_t)
	}
	\right|
	\notag\\
	&\le
	\sum_{s\ge3}
	\frac{2D\Delta^{s-2}}{k^s}\exp(s-2)
	\exp(\lambda s)
	\notag\\
	&\le 
	\sum_{s\ge3}
	\frac{2D\exp(2\lambda )}{k^2}
	\left(\frac{1}{r}\right)^{s-2}
	\left(\exp(1+\lambda )\right)^{s-2}
	\notag\\
	&=
	\frac{2D\exp(2\lambda )}{k^2}
	\sum_{j\ge1}
	\left(
	\frac{\exp(1+\lambda )}{r}
	\right)^j\notag\\
	&=\frac{2D\exp(2\lambda )}{k^2}\frac{x}{1-x},
	}
where the last three steps follow from $k\ge r\Delta$ and the conditions in (\ref{conditions}).

Moreover, by (\ref{conditions}),
we have $\exp(2\lambda)\cdot 
\frac{1+x}{1-x}\le (2-\epsilon)$,
and thus 
\aln{eq:total-error}
{
	Q_1+Q_2
	&\le
	\frac{D}{k^2}
	\left(
	\exp(2\lambda )-1
	+
	2\exp(2\lambda)\frac{x}{1-x}
	\right)
	\nonumber \\
&=	\frac{D}{k^2}
\left(\frac{\exp(2\lambda)(1+x)}{1-x}-1\right)
	\le \frac{(1-\epsilon)D}{k^2}.
}
Thus by (\ref{ine3-21}) and (\ref{eq:total-error}),
$$
\left|
\frac{\varphi'(t)}{\varphi(t)}-\frac{D}{k^2}\right|
\le Q_1+Q_2\le \frac{(1-\epsilon)D}{k^2},~~\text{and}~~\frac{\varphi'(t)}{\varphi(t)}
	\ge\frac{D\epsilon}{k^2}.
$$
Then since $\varphi(t)>0$,
\aln{bound}{
 \varphi'(t)\ge \varphi(t)
\frac{D\epsilon}{k^2}.
}

Finally, for $0\le t\le 1$, define $$\psi(t)=\exp\left(-\frac{D\epsilon }{k^2}t\right)\varphi(t).$$
Then by (\ref{bound}),
$$
\psi'(t)=\exp\left(-\frac{D\epsilon}{k^2}t\right)\left(\varphi'(t)-\frac{D\epsilon}{k^2}\varphi(t)\right)\ge 0,
$$
which implies that $\psi(t)$ is non-decreasing for $0\le t\le 1$. Thus
$$
\exp\left(-\frac{D\epsilon}{k^2}\right) \varphi(1)= \psi(1)\ge \psi(0)=\varphi(0),
$$
implying that
$$
\varphi(1)\ge 
\varphi(0)\exp\!\left(\frac{D\epsilon}{k^2}\right).
$$
The result is proven.
\end{proof}

\noindent\textit{Proof of Theorem~\ref{main1}}. \label{pro-main1}
Suppose that $G=(V,E)$ is a connected graph with maximum degree $\Delta\ge 3$, and $L$ is a $k$-assignment of $G$ with $k\ge 23.41\Delta$. 

Let $\epsilon_1 = 2\cdot 10^{-6}$, and 
let $x$ and $\lambda$ be two numbers satisfying the following two equations:
\begin{equation}
	\label{eq-x-la}
\left \{
\begin{array}{l}
\exp(\lambda)
\cdot \frac{x}{1-x}=\lambda,
\\
\exp(2\lambda)
\cdot \frac{1+x}{1-x}
=2-\epsilon_1.
\end{array}
\right.
\end{equation}
The system of 
equations in (\ref{eq-x-la})
has a solution $(x,\lambda)$
with 
$0.142295<x< 0.142296$ and 
$0.2033<\lambda<0.20331$.

Now take $\epsilon = 10^{-6}<\epsilon_1$ 
and $r=\frac{\exp(1+\lambda)}{x}$.
Then, the four numbers $\epsilon, x, \lambda$ and $r$ 
satisfy the conditions in 
(\ref{conditions}).
Note that  
$$
r=\frac{\exp(1+\lambda)}{x}
<\frac{\exp\left(1+ 0.20331\right)}{ 0.142295}
<23.41.
$$

Assume that $k$ is any integer 
with $k\ge 23.41\Delta$.
Then $k>r\Delta$. 
Obviously, 
(\ref{main1-e1}) follows directly from
Proposition~\ref{final},
and 
(\ref{main1-e1})
implies that 
$P_{\ell}(G,k)\ge P(G,k)$.
Hence, by (\ref{listcolor}), 
$P(G,k)=P_{\ell}(G,k)$
for every integer $k\ge 23.41\Delta$.
The proof is complete.
 \qed

\medskip
\noindent\textbf{Declaration on the use of AI:}
During the preparation of this work, the authors used AI systems to assist in exploring potential approaches to the proof of Theorem~\ref{main1}. All AI-generated suggestions were
verified and refined by the authors, who take full responsibility for the correctness and originality of the paper.

\noindent {\bf Data availability statement}:  Not applicable.

\section*{Acknowledgment}
This work is supported by the National Natural Science Foundation of China (No. 12501498), the Natural Science Foundation of Xiamen, China (No. 3502Z202571026), and the Foundation for Cultivated Young Talents of Fujian Province, China (No. 2025350064).


\begin{thebibliography}{99}
\bibitem{birk}
G.D. Birkhoff,
A determinant formula for the number of ways of coloring a map,
\textit{Ann. of Math.} \textbf{14} (1912), 42--46.

\bibitem{birk2}
G.D. Birkhoff and D.C. Lewis,
Chromatic polynomials,
\textit{Trans. Amer. Math. Soc.} \textbf{60} (1946), 355--451.

\bibitem{dong1}
F.M. Dong and K.M. Koh,
Foundations of the chromatic polynomial,
in J.A. Ellis-Monaghan and I. Moffatt (eds.),
\textit{Handbook of the Tutte Polynomial and Related Topics},
CRC Press, Boca Raton, FL, 2022, pp. 213--251.

\bibitem{dong0}
F.M. Dong, K.M. Koh and K.L. Teo,
\textit{Chromatic Polynomials and Chromaticity of Graphs},
World Scientific, Singapore, 2005.

\bibitem{Dong22}
F.M. Dong and M.Q. Zhang,
An improved lower bound of $P(G,L)-P(G,k)$ for $k$-assignments $L$,
\textit{J. Combin. Theory Ser. B} \textbf{161} (2023), 109--119.

\bibitem{Donner}
Q. Donner,
On the number of list-colorings,
\textit{J. Graph Theory} \textbf{16} (1992), 239--245.

\bibitem{erdos}
P. Erd\H{o}s, A.L. Rubin and H. Taylor,
Choosability in graphs,
\textit{Congr. Numer.} \textbf{26} (1979), 125--157.

\bibitem{Jackson2015}
B. Jackson,
Chromatic polynomials,
in L.W. Beineke and R.J. Wilson (eds.),
\textit{Topics in Chromatic Graph Theory},
Encyclopedia of Mathematics and its Applications, vol. 156,
Cambridge University Press, Cambridge, 2015, pp. 56--72.

\bibitem{Kaul23}
H. Kaul, A. Kumar, A. Liu, J.A. Mudrock, P. Rewers, P. Shin,
M.S. Tanahara and K. To,
Bounding the list color function threshold from above,
\textit{Involve} \textbf{16} (2023), 849--882.

\bibitem{Kaul22}
H. Kaul, A. Kumar, J.A. Mudrock, P. Rewers, P. Shin and K. To,
On the list color function threshold,
\textit{J. Graph Theory} \textbf{105} (2024), 386--397.


\bibitem{Kirov2016} R. Kirov and R. Naimi, List coloring and $n$-monophilic graphs, \textit{Ars Comb.} \textbf{124} (2016), 329--340.

\bibitem{kosto}
A.V. Kostochka and A.F. Sidorenko,
Problem session of the Prachatice conference on graph theory,
in J. Ne\v{s}et\v{r}il and M. Fiedler (eds.),
\textit{Fourth Czechoslovakian Symposium on Combinatorics, Graphs and Complexity},
Annals of Discrete Mathematics, vol. 51,
North-Holland, Amsterdam, 1992, p. 380.

\bibitem{penrose}
O. Penrose, Convergence of fugacity expansions for classical systems, in Statistical
Mechanics: Foundations and Applications (T. A. Bak, ed.), Benjamin, New York/Amsterdam, 1976, pp. 101--109.

\bibitem{rea1}
R.C. Read,
An introduction to chromatic polynomials,
\textit{J. Combin. Theory} \textbf{4} (1968), 52--71.

\bibitem{rea2}
R.C. Read and W.T. Tutte,
Chromatic polynomials,
in L.W. Beineke and R.J. Wilson (eds.),
\textit{Selected Topics in Graph Theory 3},
Academic Press, New York, 1988, pp. 15--42.

\bibitem{Sokal}
A.D. Sokal,
Bounds on the complex zeros of (di)chromatic polynomials and
Potts-model partition functions,
\textit{Combin. Probab. Comput.} \textbf{10} (2001), 41--77.

\bibitem{Stone1985}
A.H. Stone,
Trees and power-sums,
\textit{Amer. Math. Monthly} \textbf{92} (1985), 328--331.

\bibitem{Thom}
C. Thomassen,
The chromatic polynomial and list colorings,
\textit{J. Combin. Theory Ser. B} \textbf{99} (2009), 474--479.

\bibitem{viz}
V.G. Vizing,
Coloring the vertices of a graph in prescribed colors,
\textit{Diskret. Analiz}, No. 29,
Metody Diskret. Anal. v Teorii Kodovi Skhem \textbf{101} (1976), 3--10.
	

\bibitem{wangwei}
W. Wang, J. Qian and Z. Yan,
When does the list-coloring function of a graph equal its chromatic polynomial,
\textit{J. Combin. Theory Ser. B} \textbf{122} (2017), 543--549.

\bibitem{Whi1932}
H. Whitney,
A logical expansion in mathematics,
\textit{Bull. Amer. Math. Soc.} \textbf{38} (1932), 572--579.
\end{thebibliography}
\end{document}